\documentclass[a4paper,12pt]{article}	% these are the standard options,change article to book or letter as needed

\usepackage[utf8x]{inputenc}     % the kind of keyboard you have, depends on operating system and language of keyboard
\usepackage[british]{babel}      % tell LaTeX what language your document is written in, standard is American English

\usepackage{amsmath}  % maths-improvements
\usepackage{amsfonts} % maths-fonts
\usepackage{amssymb}  % maths-symbols
\usepackage{amsthm}   % theorem style

\usepackage{graphicx} % to include graphics

\usepackage{comment} % gives the comment environment

\usepackage{ae}      % forces LaTeX to use scalable fonts

\newcommand{\N}{\mathbb{N}}

\newcommand{\R}{\mathbb{R}}

\renewcommand{\le}{\leqslant}
\renewcommand{\ge}{\geqslant}

\newcommand{\ra}{\rightarrow}

\usepackage{enumerate}    % different enumerations

\newtheoremstyle{example}
     {\topsep}{\topsep} % measure of space to leave above and below the theorem
     {\itshape}%         Body font 
     {}%         Indent amount (empty = no indent, \parindent = para indent)
     {\bfseries}% Thm head font
     {.}%        Punctuation after thm head
     {\newline}%     Space after thm head (\newline = linebreak)
     {\thmname{#1}\thmnumber{ #2}\thmnote{ #3}}%         Thm head spec

\theoremstyle{example} % means the \newtheorem command uses this style

\newtheorem{theorem}{Theorem}[section]            %optional argument [section] means numbering follows the section
\newtheorem{lemma}[theorem]{Lemma}

\newtheorem{definition}[theorem]{Definition}

\newtheorem{assumption}[theorem]{Assumption}

\theoremstyle{remark}
\newtheorem{remark}[theorem]{Remark}

\newtheorem{example}[theorem]{Example}
\title{heat trace asymptotics for quantum graphs}
\author{Ralf Rueckriemen}
\date{\today} % this is the default, if you don't put anything Latex will use today

\begin{document}
%%%%%%%%%%%%%%%%%%%%%%%%%%%%%%%%%%%%%%%%%%%%%%%%%%%%%%%%%%%%%%%%%%%%%%%%%%%%%%%%%%%%%%%%%%%%%%%%%%%%%%%%%%%

\maketitle      % this generates the title, what it looks like depends on the document type

 % \begin{comment}
\begin{abstract}
We consider a quantum graph where the operator contains a potential. We show that this operator admits a heat kernel. Under some assumptions on the potential, this heat kernel admits an asymptotic expansion at $t=0$ with coefficients that depend on the potential in a universal way. These coefficients are spectral invariants, we compute the first few of them.
\end{abstract}
 % \end{comment}

% \tableofcontents

\section{Introduction}
%%%%%%%%%%%%%%%%%%%%%%%%%%%%%%%%%%%%%%%%%%%%%%%%%%%%%%%%%%%%%%%%%%%%%%%%%%%%%%%%%%%%%%%%%%%%

A metric graph is a combinatorial graph where each edge is equipped with a length. A pair of a metric graph together with a differential operator  acting on it, is called a quantum graph. The operator is of the form Laplacian plus lower order parts on each individual edge. Some boundary conditions are imposed at the vertices to make the operator self adjoint. 

Quantum graphs were introduced in the 1930s in the physics literature, they are a popular model for various processes involving wave propagation. If the operator has no potential, the eigenvalue equation on each edge can be solved explicitly, the eigenfunctions are sine waves. In this setting one can write down an exact trace formula. It relates the set of eigenvalues to topological properties of the graph. The first such trace formula was found by \cite{Roth83}, since then various generalizations have been shown, see  \cite{KottosSmilansky99, KPS07, BolteEndres09} or \cite{BolteEndres08} for a survey. One can get an exact trace formula for an operator with potential, however it uses the 
abstract existence of solutions to the eigenvalue equation on each edge as input \cite{RueckriemenSmilansky12}. Thus without specific knowledge of what these solutions look like, one cannot read off any information about the graph or the potential from this trace formula. 

For this reason we consider the heat kernel. Let $G$ be a quantum graph with the operator $D:= -\left(\frac{\partial^2}{\partial x^2}+U(x)\right)$ acting on it, here $U$ is a real potential function. Non-Robin boundary conditions are imposed at all vertices, see definition \ref{boundary_conditions}. The heat kernel $e(t,x,y)$ is a function that satisfies 
\begin{align*}
 (\partial_t+D_y)e(t,x,y)=0 && \lim_{t \ra 0} e(t,x,y)= \delta_{x}(y)
\end{align*}
where $\delta$ is the Dirac-$\delta$-distribution. We will study its asymptotic expansion for time close to zero.

In a setting without potential, this approach has the immediate disadvantage of being only asymptotic compared to the exact trace formula. However, as we will show, the coefficients in the expansion depend on the metric graph and the potential in a universal way and can be computed explicitly. This means they give rise to a series of spectral invariants such as the total edge length of the graph, the integral of the potential over the graph, as well as the sum of the values at the vertices, weighted with some factor depending on the boundary conditions. 

\begin{comment}
If one applies the asymptotic expansion to an operator with no potential, the expansion breaks off after the two terms. These two terms match with the exact trace formulae in this setting (see for example \cite{KottosSmilansky99}). The remaining terms in the exact trace formulae correspond to the periodic orbits, they decay as $O(t^{\infty})$ and are thus not visible in the asymptotic expansion.
\end{comment}

More specifically we get the following two theorems. 
\begin{theorem}
Assume the potential function $U$ is in $L^2(G,\R)$, then the operator $D$ admits a heat kernel $e(t,x,y)$ and this heat kernel is unique. 
\end{theorem}

\begin{theorem}
Assume that $U$ is smooth on the edges, that $U^{(2l+1)}(v)=0$ for all vertices $v \in V$ and $l \in \N$ and that $U^{(2l)}(x)$ is a continuous function on the entire graph $G$ for all $l \in \N$, then this heat kernel admits an asymptotic expansion of the form
 \begin{align*}
  \int_G e(t,x,x) dx \sim_{t \ra 0^+} 
  \frac{1}{\sqrt{4\pi t}} \sum_{n=0}^{\infty}\int_G a_n(x) dx \cdot t^n 
 + \frac{1}{4} \sum_{v \in V}\sum_{\alpha \sim v} \sigma_v^{\alpha\alpha} \sum_{n=0, \frac{1}{2},1,\frac{3}{2}, \hdots}a_n^b(v)\cdot t^n
 \end{align*}
 Here the $\sigma_v^{\alpha\alpha}$ come from the boundary conditions, see \ref{boundary_conditions}. The coefficients $a_n(x)$ and $a_n^b(v)$ are local and universal. With our assumptions on the potential, $a_n^b(v)=0$ whenever $n$ is not an integer. The first few non-zero coefficients are as follows.
\begin{align*}
& a_0(x) = 1 && a_0^b(v)=1\\
& a_1(x) = U(x) && a_1^b(v)=U(v)\\
& a_2(x) =  \frac{1}{6}U''(x)+\frac{1}{2}U(x)^2 && a_2^b(v)=\frac{1}{4}U''(v)+\frac{1}{2}U(v)^2\\
& a_3(x) = \frac{1}{60}U^{(4)}(x)+\frac{1}{6}U''(x)U(x)+ \frac{1}{12}U'(x)^2+\frac{1}{6}U(x)^3  \\
& a_3^b(v)=\frac{1}{32}U^{(4)}(v)+\frac{1}{4}U(v)U''(v)+\frac{1}{6}U(v)^3
\end{align*}
\end{theorem}

This situation closely mirrors the case of Riemannian manifolds where similar questions have been studied for quite some time. One of the earliest results is the work by Minakshisundaram and Pleijel in the 1950s, \cite{MinakshisundaramPleijel49, Minakshisundaram53}. They show the existence and universality of the asymptotic expansion for Riemannian manifolds without boundary and compute the coefficients up to $a_2$ for the standard Laplacian, see also \cite{McKeanSinger67, BGM71}. Their approach uses a parametrix, this is the easiest way to start out but has serious drawbacks if one wants to consider manifolds with boundary and the computation of more coefficients gets prohibitively complex. 

Existence of a heat kernel and an asymptotic expansion with universal coefficients has been shown for Riemannian manifolds with boundary, arbitrary elliptic operators and an extensive set of possible boundary conditions in \cite{Greiner71} using a method based on the inversion of the operator on the level of symbols. Using this broad existence result, \cite{Gilkey79} computed the coefficients up to $a_3$ and $a_3^b$ for manifolds with boundary and operators involving a potential using recursion relations.

We will first show existence of the heat kernel for a quantum graph using properties of the eigenfunctions. While this method is very general, it seems unsuitable to study further properties of the heat kernel. The heat kernel is a local object, something that is rather difficult to see from the eigenfunction construction as eigenfunctions and eigenvalues are highly non-local.

We will therefore use a different technique, namely a parametrix. This comes at the cost of some restrictions on the allowable potential functions but in return gives us much more detailed information about the heat kernel. Using a parametrix makes it obvious that the heat kernel is a local object. The parametrix approach also gives us the existence of the asymptotic expansion essentially for free and provides a method to compute the coefficients.

The construction of a parametrix for the heat kernel on a quantum graph is done in three steps. First, we look at the construction on the real line. Second, we use the construction on the real line to build a parametrix on a star graph, ie multiple half-infinite edges with a common central vertex. Finally, we use a partition of unity argument to construct the parametrix for a graph. 

Once we have build a parametrix for the graph we can build a heat kernel from it with a method essentially analogous to the manifold setting. 
The parametrix approximates the heat kernel, so we can use it to show the existence of an asymptotic expansion and compute the first few coefficients.

This paper is organized as follows. Section \ref{section:setup} contains the general setup and the definitions of the boundary conditions we are going to use. Section \ref{section:existence_heat_kernel} contains the general existence proof of the heat kernel. Next, section \ref{section:parametrix} contains the construction of a parametrix in the three steps outlined above. In section \ref{section:heat_kernel} we build the heat kernel for the graph from a parametrix. Finally, in section \ref{section:asymptotics} we show the existence of an asymptotic expansion and compute the coefficients.

\section{Setup}
%%%%%%%%%%%%%%%%%%%%%%%%%%%%%%%%%%%%%%%%%%%%%%%%%%%%%%%%%%%%%%%%%%%%%%%%%%%%%%%%%%%%%%%%%%%%
\label{section:setup}

Let $G$ be a metric graph. The number of edges and vertices is finite and the length of each edge is finite. Let the operator $D:= -\left(\frac{\partial^2}{\partial x^2}+U(x)\right)$ act on it, here $U$ is a potential function. At this stage we only assume $U \in L^2(G,\R)$ and make no assumptions on its behaviour at the vertices. The sign convention for $U$ might seem unusual but it has the advantage of avoiding any minus-signs in the coefficients in the heat asymptotics.

We will set up boundary conditions at the vertices as in \cite{KostrykinSchrader99} and make use of their classification.
\begin{definition}
\label{boundary_conditions} 
 Given a vertex $v$ of the quantum graph and a function $f$ on the graph, enumerate the edges adjacent to $v$ from $1,\hdots ,\alpha$. Denote by $\overrightarrow{f_v}$ the vector 
 $\left( f_1(v), \hdots, f_{\alpha}(v) \right)$ and by $\overrightarrow{f_v}'$ the vector $\left( f_1'(v), \hdots, f_{\alpha}'(v)\right)$ where all derivatives are oriented away from the vertex. Then the boundary conditions can be written as
\begin{align*}
 A_v\overrightarrow{f_v} + B_v\overrightarrow{f_v}' = 0
\end{align*}
for two matrices $A_v$ and $B_v$. Let
\begin{align*}
 \sigma^v := -(A_v+ikB_v)^{-1}(A_v-ikB_v)
\end{align*}
for $k \in \R$. If the matrix $\sigma^v$ is $k$-independent, the boundary conditions are said to be non-Robin. This means, they don't mix conditions on the function with conditions on its derivative. Note that $\sigma^v$ is unitary and that $(\sigma^v)^2=Id$, \cite{FKW07}.

Collecting the matrices $A_v$ and $B_v$ for all the vertices as blocks along the diagonal in the matrices $A$ and $B$, the operator $D$ is self-adjoint if and only if $AB^*$ is self adjoint and $(A,B)$ has full rank.

We will henceforth assume that non-Robin boundary conditions are imposed at all vertices.
\end{definition}

\begin{example}
The Kirchhoff-Neumann boundary conditions at a vertex correspond to the matrix 
\begin{align*}
 (\sigma_{KN})^v_{\alpha \beta}=\frac{2}{\deg(v)} - \delta_{\alpha \beta} 
\end{align*}
where $\delta_{\alpha \beta}$ is the Kronecker-$\delta$. 
\end{example}

\begin{remark}
 One could consider a more general operator of the form 
 \begin{align*}
  D_{\mathcal{A}}:=  \left(i\frac{\partial}{\partial x} + \mathcal{A}(x)\right)^2 - U(x)
 \end{align*}
with an additional magnetic potential $\mathcal{A}$ but this magnetic potential can be completely absorbed in the boundary conditions through a gauge transformation, see \cite{KostrykinSchrader03}. 

We have $D_{\mathcal{A}}= \mathcal{U}^{-1}D_0\mathcal{U}$ where $\mathcal{U}$ is unitary. The boundary conditions are transformed as $A \mapsto A\mathcal{U}$ and $B \mapsto B\mathcal{U}$, in particular if $D_{\mathcal{A}}$ has non-Robin boundary conditions at all vertices then so has $D_0$.

The eigenfunctions $f_{\mathcal{A}}$ of the operator $D_{\mathcal{A}}$ can be obtained from the ones of $D_0$ by multiplying with a prefactor
\begin{align*}
 f_{\mathcal{A}}(x)= e^{-i\int_{x_0}^x\mathcal{A}(x')dx'}f_0(x)
\end{align*}
Thus the only change in the asymptotics of the heat kernel would be this prefactor, we will therefore only consider operators without magnetic potential.
\end{remark}

\begin{definition}
 We write $f(t)=O(t^k)$ if there exist constants $T$ and $C$ such that
\begin{equation*}
 |f(t)| \le Ct^k \hspace{1 cm} \forall t\in [0,T]
\end{equation*}
We will also write $f(t)=O(t^{\infty})$ if $f(t)=O(t^k)$ for all $k \ge 0$.
Note that $t^{k'}=O(t^k)$ if and only if $k'\ge k$. If the function $f$ depends on multiple variables we assume $C$ and $T$ are independent of the other variables.
\end{definition}

\section{Existence of the heat kernel}
%%%%%%%%%%%%%%%%%%%%%%%%%%%%%%%%%%%%%%%%%%%%%%%%%%%%%%%%%%%%%%%%%%%%%%%%%%%%%%%%%%%%%%%%%%%%
\label{section:existence_heat_kernel}

In this section we will prove existence of the heat kernel in a very general setting, the proof technique is specific to quantum graphs, it works because so much is known about the eigenvalue equation on the unit interval.

\begin{definition}
 Let $\{ k_j^2, \theta_j \}_j$ be a spectral resolution of $D$. The $k_j^2$ are the eigenvalues and the eigenfunctions $\theta_j$ form an orthonormal basis of $L^2(G,\R)$. 
\end{definition}

The existence of such a spectral resolution is well known, it follows from the fact that $D$ is elliptic and the graph is finite, the proof is written out in \cite{Kuchment04} for the case of an operator without potential but it immediately generalizes to our setting.

Consider an individual edge without boundary conditions, parametrized as the interval $[0,L]$, then the equation
\begin{align*}
 -y''(x) - U(x)y(x) = k^2y(x)
\end{align*}
has two real linearly independent solutions $y_1^k$ and $y_2^k$ for any fixed value of $k$, \cite{PoeschelTrubowitz87}. If the potential is only in $L^2[0,L]$ these solutions are in $C^1[0,L]$ with an absolutely continuous derivative but not necessarily in $C^2[0,L]$. Normalize them so that
\begin{align*}
 y_1^k(0) = 1 && (y_1^k)'(0)=0 \\
 y_2^k(0) = 0 && (y_2^k)'(0)=k
\end{align*}
inspired from $\cos(kx)$ and $\sin(kx)$ in the no potential case. 

\begin{theorem}
 \cite{PoeschelTrubowitz87}
\label{individual_bounds}
 The two solutions $y_1^k$ and $y_2^k$ can be written as
 \begin{align*}
  y_1^k(x) = \cos(kx) + E_1^k(x) && y_2^k(x) =\sin(kx) +E_2^k(x)
 \end{align*}
where the two remainder functions $E_1^k$ and $E_2^k$ satisfy
\begin{align*}
 |E_{1,2}^k(x)| \le \frac{1}{k}e^{||U||L}
\end{align*}
\end{theorem}

The eigenfunctions on each edge of the graph are then linear combinations of these two solutions, which we write as
\begin{align*}
 \theta_j(x) = \rho_j \left(\cos(\alpha_j)y_1^{k_j}(x)+ \sin(\alpha_j)y_2^{k_j}(x) \right)
\end{align*}
for some real parameters $\rho_j$ and $\alpha_j$. To get an eigenfunction on the entire graph one needs to pick the coefficients $\rho_j$ and $\alpha_j$ on each edge in such a way that the boundary conditions are satisfied at all vertices.

\begin{remark}
 The construction above reduces the problem of finding eigenvalues and eigenfunctions to a finite list of linear equations. This, or a very similar approach, is the basis of the proof of an exact trace formula for quantum graphs, \cite{KottosSmilansky99, BolteEndres09, RueckriemenSmilansky12}.
\end{remark}

\begin{theorem}
\label{eigenfunction_bound}
 There exists a constant $C$ independent of $x \in G$ and $j$ such that
 \begin{align*}
  |\theta_j(x)| < C
 \end{align*}
 In other words, there exists a global bound on the size of all eigenfunctions of a quantum graph.
\end{theorem}
\begin{proof}
As we assumed the eigenfunctions to be orthonormal we have
\begin{align*}
 \int_0^L |\theta_j(x)|^2 dx \le \int_G |\theta_j|^2 = 1
\end{align*}
This implies
\begin{align*}
 \rho_j^2  \int_0^L  \left(\cos(\alpha_j)y_1^{k_j}(x)+ \sin(\alpha_j)y_2^{k_j}(x) \right)^2 dx &\le 1 \\
  \int_0^L  \left(\cos(k_jx-\alpha_j)+ \cos(\alpha_j)E_1^{k_j}(x)+ \sin(\alpha_j)E_2^{k_j}(x) \right)^2 dx &\le \frac{1}{\rho_j^2 } \\
\int_0^L  \cos^2(k_jx-\alpha_j)+O(\frac{1}{k}) dx &\le \frac{1}{\rho_j^2 } \\
\frac{L}{2} + O(\frac{1}{k}) &\le \frac{1}{\rho_j^2}
\end{align*}
In particular for $k$ sufficiently large
\begin{align*}
 \rho_j \le \sqrt{\frac{4}{L}}
\end{align*}
As the $k_j$ are discrete, this implies there is a global bound for the $\rho_j$'s.
 By theorem \ref{individual_bounds} there exists a constant $C'$, independent of $x$ and $k$, such that
 \begin{align*}
  |y_1^k(x)| <C' && |y_2^k(x)| <C'
 \end{align*}
This implies the claimed bound.
\end{proof}

\begin{theorem}[\cite{Grieser07} Weyl law]

\label{Weyl_law} 

Let $G$ be a quantum graph with non-Robin type boundary conditions at all vertices and spectrum $\{k_j^2\}_j$. Then one can estimate 
the number of eigenvalues in any interval $(K_0,K_1)$ by
\begin{align*} 
 \left|\# \{ j | K_0 < k_j < K_1 \}- \frac{\mathcal{L}}{\pi}(K_1-K_0)\right| < 2E
\end{align*}
where $\mathcal{L}$ is the total edge length of the graph and $E$ is the number of edges.
\end{theorem}

\begin{theorem}
\label{heat_kernel_existence}
 Let $G$ be a metric graph with the operator $D:= -\left(\frac{\partial^2}{\partial x^2}+U(x)\right)$ acting on it. Let  $\{ k_j^2, \theta_j \}_j$ be a spectral resolution of $D$. Then the heat kernel for $D$ on $G$ exists, is unique and is given by
 \begin{align*}
  e(t,x,y) := \sum_{j=0}^{\infty}e^{-k_j^2t}\theta_j(x)\theta_j(y)
 \end{align*}
 Note that if the potential is only assumed to be in $L^2(G,\R)$ then the heat kernel is not necessarily smooth in the $x$ and $y$ variables.
\end{theorem}
\begin{proof}
 This is a well known formal expression for the heat kernel. 
 The previous two theorems imply that the infinite sum converges absolutely for all $t>0$ and all $x,y \in G$. As the eigenfunctions are 
 smooth on the edges and satisfy the boundary conditions at the vertices so does the heat kernel.

 To get uniqueness we argue as in \cite{BGM71}. If $e(t,x,y)$ is a heat kernel, it can be written as
\begin{align*}
 e(t,x,y) = \sum_{j=0}^{\infty} f_j(x,t)\theta_j(y) &&\text{ with }&& f_j(x,t) = \int_G e(t,x,y) \theta_j(y)dy
\end{align*}
These coefficients $f_j$ satisfy 
 \begin{align*}
  \partial_t f_j(x,t) =& \int_G \partial_t e(t,x,y) \theta_j(y)dy\\
  =& -\int_G D_y e(t,x,y) \theta_j(y)dy\\
  =& -\int_G  e(t,x,y) D_y\theta_j(y)dy\\
  =& -k_j^2 f_j(x,t)
 \end{align*}
where we used the fact that $D$ is self-adjoint. This implies
\begin{align*}
 f_j(x,t) = g_j(x)e^{-k_j^2 t}
\end{align*}
Considering the limit $t \ra 0$ now shows $g_j(x)=\theta_j(x)$ and establishes uniqueness.
 \end{proof}

\section{Construction of a parametrix}
%%%%%%%%%%%%%%%%%%%%%%%%%%%%%%%%%%%%%%%%%%%%%%%%%%%%%%%%%%%%%%%%%%%%%%%%%%%%%%%%%%%%%%%%%%%%
\label{section:parametrix}

We will construct a parametrix in three steps, first on the real line, then on a star graph and finally on the quantum graph.

\begin{definition}
 A parametrix is a function $h_k(t,x,y)$ that satisfies the following: 
\begin{enumerate}
 \item It is smooth for $t>0$.
 \item It is smooth for $t\ge 0$ if $x\neq y$ and satisfies $h_k(t,x,y)=O(t^{\infty})$ for $y$ bounded away from $x$.
\item $(\partial_t+D_y)h_k(t,x,y)$ is smooth for $t \ge 0$ and satisfies $(\partial_t+D_y)h_k(t,x,y)=O(t^{k-\frac{1}{2}})$.
\item $ \lim_{t \ra 0} h_k(t,x,y)= \delta_x(y)$ 
\end{enumerate} 
 \end{definition}
 
We will explain in the respective sections what we mean by smoothness on the star graph and the quantum graph.

\subsection{The real line}
%%%%%%%%%%%%%%%%%%%%%%%%%%%%%%%%%%%%%%%%%%%%%%%%%%%%%%%%%%%%%%%%%%%%%%%%%%%%%%%%%%%%%%%%%%%%%%%%%%%%%%%%%%%%%%%%%%%%%% 
\label{section:real_line}

We will start the construction of a heat parametrix on the real line. 
% We will assume that the potential function $V$ is smooth and compactly supported. This will greatly simplify various convergence issues 
% and not be a real constraint later on as we want to use our construction on the real line to build a heat kernel on compact graphs.\\
Let $f(t,x,y):=\frac{1}{\sqrt{4\pi t}} e^{-\frac{(x-y)^2}{4t}}$ denote the heat kernel of the standard Laplacian $-\partial_x^2$ on $\R$.

We will use the ansatz
\begin{align}
\label{ansatz_real_line}
 h_k(t,x,y):=f(t,x,y)\sum\limits_{l=0}^{k}u_l(x,y) t^l
\end{align}
where the $u$-functions are determined recursively. 
We will always assume $k\ge 1$, so there is at least one term beyond $\frac{1}{\sqrt{4\pi t}} e^{-\frac{(x-y)^2}{4t}}$.

\begin{comment}
$\partial_t h_k(t,x,y)= f(t,x,y)\left( \left(-\frac{1}{2t}+\frac{(x-y)^2}{4t^2}\right)\sum\limits_{l=0}^{k}u_l t^l
+\sum\limits_{l=0}^{k-1}(l+1)u_{l+1}t^l\right)$\\

and\\
$\Delta_y h_k(t,x,y)=f(t,x,y)\left(\left(-\frac{(x-y)^2}{4t^2}+\frac{1}{2t}\right)\sum\limits_{l=0}^{k}u_l t^l
-\frac{x-y}{t}\sum\limits_{l=0}^{k}\partial_y u_l t^l - \sum\limits_{l=0}^{k}\partial_y^2u_l t^l\right)$\\
\end{comment}
We have
\begin{align*}
& \left(\partial_t-\partial_y^2 -U(y)\right)h_k(t,x,y)\\
=&f(t,x,y)\left(\sum\limits_{l=0}^{k-1}(l+1)u_{l+1}t^l - \frac{x-y}{t}\sum\limits_{l=0}^{k}\partial_y u_l t^l 
-\sum\limits_{l=0}^{k}\partial_y^2u_l t^l-U(y)\sum\limits_{l=0}^{k}u_l t^l\right)
\end{align*}

We order the terms by powers of $t$ and set all but the highest power to zero:
\begin{align*}
 0=&(x-y)\partial_y u_0\\
0=&(l+1)u_{l+1}+(y-x)\partial_y u_{l+1} -\partial_y^2u_l -U(y)u_l \\
&-\partial_y^2u_k-U(y)u_k \text{ is the highest order term }
\end{align*}

This leads to $u_0(x,y)=const$, we will show below in lemma \ref{real_parametrix} that $u_0(x,y)=1$. For $l\ge 1$ this leads to the transport equation.

\begin{lemma}
\label{transport_equation}
The transport equation on the real line is 
\begin{align*}
0=lu_l(x,y)+(y-x)\partial_y u_l(x,y)+ D_y u_{l-1}(x,y) 
\end{align*}
where $l \ge 1$ and $u_0(x,y)=1$. It has the solution
\begin{align*}
u_l(x,y) =-(y-x)^{-l}\int_x^y (z-x)^{l-1}D_z u_{l-1}(x,z)dz
\end{align*}
If the potential $U$ is smooth then so are all the $u$-functions. 
\end{lemma}
\begin{proof}
The homogeneous equation
\begin{align*}
 0=lu_l(x,y)+(y-x)\partial_y u_{l}(x,y)
\end{align*}
can be solved with separation of variables. It has the solution
\begin{align*}
u_l^{hom}(x,y)=c(y-x)^{-l}
\end{align*}
The inhomogeneous equation is then solved using variation of constants.
The smoothness follows from lemma \ref{analytic_continuation}, below.
 \end{proof}

\begin{remark}
 The general inhomogeneous solution consists of the particular one above plus a solution to the homogeneous equation. As we want our solution 
to be well defined for $x=y$ we have to choose the constant $c=0$ in the homogeneous equation. Thus the given solution is the only solution 
for the $u_l$ that is relevant to us.
\end{remark}

\begin{remark}
\label{local_u-functions}
 The solution to the transport equation implies that $u_l(x,y)$ only depends on the potential in some small neighbourhood around the interval from $x$ to $y$, in particular $u_l(x,x)$ only depends on a small neighbourhood of $x$. We will refer to this property by saying that the $u$-functions are local.
\end{remark}

\begin{lemma}
\label{analytic_continuation}
Let $f$ be a $C^r$ function, and let 
\begin{equation*}
 g(y)=(y-a)^{-n}\int_a^y (z-a)^{n-1}f(z)dz
\end{equation*}
then $g$ is $C^r$ at $a$ and
\begin{equation*}
 g(a)=\frac{f(a)}{n}
\end{equation*}
\end{lemma}
\begin{proof}
Use the $\varepsilon-\delta$ criterion for the $C^0$ case and a Taylor expansion for higher orders.
\end{proof}

\begin{lemma}
\label{real_parametrix}
Let $h_k(t,x,y)$ be defined as in equation (\ref{ansatz_real_line}) with the $u$-functions defined by the transport equation, lemma \ref{transport_equation}. Assume the potential $U$ is smooth. Then $h_k(t,x,y)$ is a parametrix on the real line.
\end{lemma}
\begin{proof}
The first three properties just follow from the construction and the smoothness of the $u_l$. 

For the last one recall that $f(t,x,y)=\frac{1}{\sqrt{4\pi t}} e^{-\frac{(x-y)^2}{4t}}$ is the heat kernel for the Laplacian without potential. 
Let $N_x$ be an open neighbourhood of $x$ with compact closure. Then
\begin{eqnarray*}
 &&\lim_{t\ra 0}\int_{N_x} h_k(t,x,y)\psi(y)dy\\
 &=&\lim_{t\ra 0}\int_{N_x} f(t,x,y)\sum_{l=0}^{k}u_l(x,y) t^l\psi(y)dy\\
 &=&\lim_{t\ra 0}\sum_{l=0}^{k}t^l\int_{N_x} f(t,x,y)u_l(x,y)\psi(y)dy\\
&=& \psi(x)u_0(x,x)
\end{eqnarray*}
as each integral in the sum converges to $u_l(x,x)\psi(x)$ by properties of $f$, so we have to set $u_0(x,x)=1$ .
\end{proof}

\subsection{A star graph}
%%%%%%%%%%%%%%%%%%%%%%%%%%%%%%%%%%%%%%%%%%%%%%%%%%%%%%%%%%%%%%%%%%%%%%%%%%%%%%%%%%%%%%%%%%%%%%%%%%%%%%%%%%%%%%%%%%%%%%%%%%%%%%%%
\label{section:star_graph}

A star graph consists of $d$ half-infinite edges glued together at a central vertex $v$. At this vertex non-Robin boundary conditions are imposed as in section \ref{section:setup}. 

\begin{definition}
 We say a function is smooth on the star graph if its restriction to any two edges gives rise to a smooth function.
\end{definition}

This implies that all odd derivatives of the function vanish at the central vertex, whereas the even derivatives are continuous at the vertex. 

\begin{assumption}
We will assume that the potential function $U$ is smooth on the star graph. 
\end{assumption}

\begin{remark}
Through carefully keeping track of degrees of differentiability the assumptions on the potential can be relaxed. In this case, only the lower order $u$-functions are well defined and they have low degrees of differentiability at zero. However, any approach that uses a parametrix will require at the very least that the potential is continuous and that its first derivative vanishes at all the vertices. The general proof of existence of the heat kernel in section \ref{section:existence_heat_kernel} indicates that even these conditions should not be necessary. 

We will therefore not try to find the weakest possible differentiability conditions for the parametrix approach but keep the assumption above. This will significantly simplify further computations as this assumption makes all the $u$-functions smooth. Hopefully some different approach to heat kernel asymptotics on quantum graphs will give rise to results without such strong conditions on the potential in the future.  
\end{remark}

Let $x_{\alpha}$ denote the coordinate for the edge $\alpha$, parametrized as $[0, \infty)$.

In order to model a reflection at the vertex we will continue the potential function via $U_{\alpha}(-x_{\alpha}):=U_{\alpha}(x_{\alpha})$ onto the entire real line. The assumptions on the potential make this a smooth continuation.  

Let $ h_k^{\alpha \beta}(t,x_{\alpha}, -y_{\beta})$ denote the parametrix for the real line for the potential $U_{\alpha}(x)$ for $x \ge 0$ and $U_{\beta}(-x)$ for $x \le 0$. 

We define the parametrix of a star graph as follows.

\begin{definition}
\label{star_graph}
Let
\begin{align*}
 h_k^v(t,x_{\alpha},y_{\beta}):= \delta_{\alpha \beta}h_k^{\alpha \alpha}(t,x_{\alpha}, y_{\alpha}) 
 + \sigma_{v}^{\alpha \beta} h_k^{\alpha \beta}(t,x_{\alpha}, -y_{\beta})
\end{align*}
where $\delta_{\alpha \beta}$ is the Kronecker-$\delta$ and  $\sigma_{v}^{\alpha \beta}$ is an element of the matrix with the boundary conditions, see definition \ref{boundary_conditions}. This implies that $h_k^v$ satisfies the boundary conditions at the central vertex.
\end{definition}

\begin{lemma}
\label{star_properties}
 The function $h_k^{v}(t,x,y)$ is a parametrix on the star graph.
\end{lemma}
\begin{proof}
The first three properties follow from lemma \ref{real_parametrix}. For the last one, assume that $x$ is not the central vertex and lies on the edge $\alpha$. Let $G_{\alpha}$ denote the restriction of the graph to the half edge $\alpha$ and $N_x$ an open neighbourhood of the point $x$ with compact closure. Then 
\begin{align*}
& \lim_{t \ra 0}\int_{N_{x}} h_k^{v}(t,x,y)\psi(y)dy \\
 =& \lim_{t \ra 0}\int_{N_x \cap G_{\alpha}} h_k^{\alpha \alpha}(t,x_{\alpha},y_{\alpha})\psi(y_{\alpha})dy_{\alpha}  \\
& + \lim_{t \ra 0}\sum_{\beta}\sigma_v^{\alpha \beta}\int_{N_x \cap G_{\beta}} h_k^{v}(t,x_{\alpha},-y_{\beta})\psi(y_{\beta})dy_{\beta} 
\end{align*}
The first term converges to $\psi(x_{\alpha})$ and the remaining ones all converge to zero by lemma \ref{real_parametrix}. The case of $x$ equal to the central vertex follows by continuity.
\end{proof}

\subsection{Construction of a parametrix for a quantum graph}
%%%%%%%%%%%%%%%%%%%%%%%%%%%%%%%%%%%%%%%%%%%%%%%%%%%%%%%%%%%%%%%%%%%%%%%%%%%%%%%%%%%%%%%%%%%%%%%%%%%%%%%%%%%%%
\label{section:quantum_graph}

We assume that the graph does not have loops or multiple edges. This can be achieved without loss of generality by inserting additional vertices 
of degree $2$ with Kirchhoff-Neumann boundary conditions. These will not influence the spectrum of the operator so they will not show up in the asymptotics but they make the definition of the parametrix easier.

Let $l_0$ denote the length of the shortest edge on the graph and let $d$ denote the distance function on the graph.

\begin{definition}
Let $\eta : \R \ra [0, 1]$ be a smooth cut-off function such that
\begin{itemize}
 \item $\eta\mid_{(-\infty,\frac{1}{3}]}=1$ and  $\eta\mid_{[\frac{2}{3},\infty)}=0$
 \item $\eta(1-x)=1-\eta(x)$, that is $\eta$ is symmetric about the point $(\frac{1}{2},\frac{1}{2})$
\end{itemize}

\end{definition}

\begin{definition}
Let $V^v_{\frac{l_0}{3}}:= \{ x \in G \mid d(x,v) \le \frac{l_0}{3} \}$ be a neighbourhood of the vertex $v$.
Let $V_{\frac{l_0}{3}}:=\cup_{v\in V}V^v_{\frac{l_0}{3}}$ and denote the complement of this set in $G$ by $V_{\frac{l_0}{3}}^c$.
\end{definition}

\begin{definition}
We define a partition of unity $\{\chi_v \}_{v\in V}$ on the graph $G$ as follows. 
For $x$ on an edge adjacent to $v$ set 
\begin{align*}
 \chi_v(x):=\eta( l(x)^{-1}d(x,v))
\end{align*}
where $l(x)$ is the length of that edge. Set $\chi_v(v')= \delta_{vv'}$ for vertices and $\chi_v(x)=0$ if $x$ is not on an edge adjacent to $v$. This implies $\chi_v\mid_{V^v_{\frac{l_0}{3}}}\equiv 1$.
\end{definition}

For the potential function, we will just carry over the smoothness assumption from the star graph to an arbitrary graph.

\begin{assumption}
 We assume that the potential function $U$ is smoooth on the graph. In particular
\begin{equation*}
 U_{\alpha}^{(2l+1)}(v)=0 \hspace{1cm} and \hspace{1cm} U_{\alpha}^{(2l)}(v)=U_{\alpha'}^{(2l)}(v)
\end{equation*}
for all edges $\alpha, \alpha'$ adjacent to $v$ for all vertices $v$, for all $l \in \N$. 
\end{assumption}
This guarantees that the $u$-functions are well defined and their restrictions to any two adjacent edges are smooth functions.

\begin{definition}
\label{qgraph_parametrix}
Let $x,y \in G$, then we set
\begin{equation*}
 \tilde{h}_k(t,x,y):= \eta(2l_0^{-1}d(x,y))\sum_{v\in V}\chi_v(x)h_k^v(t,x,y)
\end{equation*}
where $h_k^v$ is the parametrix on the star graph constructed in definition \ref{star_graph}. 
\end{definition}

\begin{remark}
Note that this is well defined in the sense that the factor in front of $h_k^v(t,x,y)$ is only non-zero if both $x$ and $y$ lie on edges adjacent 
to $v$. 

The notation here is global and coordinate free. The coordinates only come into play at the level of the star graphs. If a point lies on an edge between the vertices $v$ and $v'$ this edge will be parametrised from $v$ to $v'$ in $h_k^v$ but from $v'$ to $v$ in $h_k^{v'}$ so the same point has different coordinates in different elements of the sum.
\end{remark}

\begin{lemma}
 The function $ \tilde{h}_k$ satisfies the boundary conditions at all the vertices.
\end{lemma}
\begin{proof}
If $d(x,y) > \frac{l_0}{3}$ then $\tilde{h}_k(t,x,y)=0$ because of the initial cut-off function. If $d(x,y) \le \frac{l_0}{3}$ and $x$ and $y$ are both within a $V^v_{\frac{l_0}{3}}$ neighbourhood of a vertex $v$ we have $\tilde{h}_k(t,x,y)=h_k^v(t,x,y)$ which satisfies the boundary conditions by construction, see definition \ref{star_graph}. 
\end{proof}

\begin{lemma}
\label{qgraph_properties}
 The function $\tilde{h}_k(t,x,y)$ is a parametrix on the graph $G$.
\end{lemma}
\begin{proof}
The first two properties follow directly from the analogous properties of $h_k^v$ in lemma \ref{star_properties}.
For the third one we have
\begin{align*}
 &(\partial_t+D_y)\tilde{h}_k(t,x,y)\\
 =&  \eta(2l_0^{-1}d(x,y))\sum_{v\in V}\chi_v(x)(\partial_t+D_y)h_k^v(t,x,y)\\
& - \partial_y(\eta(2l_0^{-1}d(x,y)))\sum_{v\in V}\chi_v(x)\partial_yh_k^v(t,x,y) \\
&- \partial_y^2(\eta(2l_0^{-1}d(x,y)))\sum_{v\in V}\chi_v(x)h_k^v(t,x,y)
\end{align*}
The first term extends smoothly to $t\ra 0$ and satisfies $(\partial_t+D_y)h_k^v(t,x,y)= O(t^{k-\frac{1}{2}})$ by lemma \ref{star_properties}. The second and third term are only non-zero if $\frac{l_0}{6} \le d(x,y) \le \frac{l_0}{3}$ because the initial cut-off function is constant otherwise.
In this region $h_k^v(t,x,y)$ extends smoothly to $t \ra 0$ and satisfies $h_k^v(t,x,y)=O(t^{\infty})$ by lemma \ref{star_properties}.

Finally, if $x$ lies in the $V^v_{\frac{l_0}{3}}$ neighbourhood of a vertex $v$ this property follows from lemma \ref{star_properties}. If $x$ lies in $V^c_{\frac{l_0}{3}}$ on an edge between two vertices $v$ and $v'$ this again follows from \ref{star_properties} and $\chi_v(x)+\chi_{v'}(x)=1$ because this is a partition of unity.
\end{proof}

\section{Construction of the heat kernel}
%%%%%%%%%%%%%%%%%%%%%%%%%%%%%%%%%%%%%%%%%%%%%%%%%%%%%%%%%%%%%%%%%%%%%%%%%%%%%%%%%%%%%%%%%%%%%%%%%%%%%%%%%%%%%
\label{section:heat_kernel}

% \subsection{Heat kernel estimates}
%%%%%%%%%%%%%%%%%%%%%%%%%%%%%%%%%%%%%%%%%%%%%%%%%%%%%%%%%%%%%%%%%%%%%%%%%%%%%%%%%%%%%%%%%%%%%%%%%%%%%%%%%%%%%

\begin{definition}
The convolution of two continuous kernels $P, Q \in C^0(\R_{>0}, G, G)$ on a metric graph $G$ is defined as follows.
\begin{equation*}
 (P*Q)(t,x,y):=\int_0^t \int_{G}P(s,x,z)Q(t-s,z,y)dzds
\end{equation*}
\end{definition}

\begin{lemma}
\label{convolution_estimate} 
Suppose $P(t,x,y)=O(t^k)$ and $Q(t,x,y)=O(t^{k'})$, with $k,k' > -1$, then
\begin{equation*}
 (P*Q)(t,x,y) = O(t^{k+k'+1})
\end{equation*}
\end{lemma}
\begin{proof}
This follows from a direct computation.
\begin{align*}
 |(P*Q)(t,x,y)|\le & \int_0^t \int_{G} C_P s^{k}  C_Q (t-s)^{k'}dzds\\
=& C_PC_Q \mathcal{L} t^{k+k'+1}\int_0^1 (s')^k(1-s')^{k'}ds'\\
 =& C_P C_Q \mathcal{L} t^{k+k'+1}B(k+1,k'+1)
\end{align*}
Here $B(k+1,k'+1)$ is the Beta-function and $\mathcal{L}$ is the total edge length of the graph $G$.
\end{proof}

\begin{definition}
 Let 
\begin{eqnarray*}
 \tilde{g}_k(t,x,y):=(\partial_t+D_y)\tilde{h}_k(t,x,y)
\end{eqnarray*}
\end{definition}

\begin{lemma}
\label{heat_operator_convolution}
Let $P$ be a continuous kernel, then
\begin{equation*}
(\partial_t + D_y ) (P*\tilde{h}_k)=P+P*\tilde{g}_k
\end{equation*}
\end{lemma}
\begin{proof}
Proof adapted from \cite{Rosenberg97}.
 \begin{eqnarray*}
  &&(\partial_t + D_y)(P*\tilde{h}_k)(t,x,y)\\
&=&\partial_t \int_0^t \int_{G}P(s,x,z)\tilde{h}_k(t-s,z,y)dzds+ D_y(P*\tilde{h}_k)(t,x,y)\\
&=& \lim_{s\ra t} \int_{G}P(s,x,z)\tilde{h}_k(t-s,z,y)dz + \int_0^t \int_{G}P(s,x,z)\partial_t \tilde{h}_k(t-s,z,y)dzds\\
&& + \int_0^t \int_{G}P(s,x,z)D_y\tilde{h}_k(t-s,z,y)dzds\\
&=& P(t,x,y) + (P*\tilde{g}_k)(t,x,y)
 \end{eqnarray*}
where we used the fact that $\tilde{h}_k(0,x,y)=\delta_{x}(y)$ in the last step.
\end{proof}

\begin{lemma}
\label{graph_estimate}
 We have
\begin{align*}
 \sum_{l=1}^{\infty}(-1)^{l+1}\tilde{g}_k^{*l}(t,x,y) = O(t^{k-\frac{1}{2}})
\end{align*}
in particular, this infinite sum converges.
\end{lemma}
\begin{proof}
Proof adapted from \cite{Rosenberg97}. By lemma \ref{qgraph_properties}
\begin{equation*}
 |\tilde{g}_k(t,x,y)| \le Ct^{k-\frac{1}{2}} \le CT^{k-\frac{1}{2}}=:\overline{C}
\end{equation*}
for all $0 \le t \le T$ for some $T$. 
We will now show by induction 
\begin{equation*}
 |\tilde{g}_k^{*l}(t,x,y)| \le  \frac{C\overline{C}^{l-1}\mathcal{L}^{l-1}t^{k-\frac{1}{2}+l-1}}{(k-\frac{1}{2}+1)(k-\frac{1}{2}+2)\hdots(k-\frac{1}{2}+l-1)}
\end{equation*}
we have just shown the case $l=1$. We have
\begin{eqnarray*}
&& |\tilde{g}_k^{*l}(t,x,y)|\\
&\le& \int_0^t\int_G |\tilde{g}_k^{*(l-1)}(s,x,z)|\cdot|\tilde{g}_k(t-s,z,y)|dzds\\
&\le& \int_0^t\int_G \frac{C\overline{C}^{l-2}\mathcal{L}^{l-2}s^{k-\frac{1}{2}+l-2}}{(k-\frac{1}{2}+1)(k-\frac{1}{2}+2)\hdots(k-\frac{1}{2}+l-2)}\overline{C} dzds\\
&\le& \frac{C\overline{C}^{l-1}\mathcal{L}^{l-1}}{(k-\frac{1}{2}+1)(k-\frac{1}{2}+2)\hdots(k-\frac{1}{2}+l-2)}\int_0^ts^{k-\frac{1}{2}+l-2}ds
\end{eqnarray*}
which finishes the induction. Applying the ratio test shows that the sum of the upper bounds converges, this implies that the original sum converges as well.
\end{proof}

\begin{theorem}
 Let 
\begin{equation*}
 e(t,x,y):=\tilde{h}_k(t,x,y)+ \tilde{h}_k(t,x,y)*\sum_{l=1}^{\infty}(-1)^{l}\tilde{g}_k(t,x,y)^{*l}
\end{equation*}
then $e(t,x,y) \in C^{\infty}(\R_{>0} \times G \times G)$, is independent of $k$ for $k\ge 1$ and is the heat kernel on $G$.
\end{theorem}
\begin{proof}
 By lemma \ref{graph_estimate} we know that the infinite sum converges.
 
Using lemma \ref{heat_operator_convolution} we get:
\begin{eqnarray*}
 (\partial_t+D_y)e(t,x,y)&=& \tilde{g}_k + \sum_{l=1}^{\infty}(-1)^{l}\tilde{g}_k(t,x,y)^{*l}+\tilde{g}_k*\sum_{l=1}^{\infty}(-1)^{l}\tilde{g}_k(t,x,y)^{*l}\\
&=&0
\end{eqnarray*}
We also have
\begin{eqnarray*}
&& \int_G e(t,x,y)\psi(y)dy\\
&=& \int_G \tilde{h}_k(t,x,y)\psi(y)dy+ \int_G\tilde{h}_k*\sum_{l=1}^{\infty}(-1)^{l}\tilde{g}_k(t,x,y)^{*l}\psi(y)dy\\
&\ra_{t\ra 0}& \psi(x)+0
\end{eqnarray*}
The second term doesn't contribute because we have 
\begin{align*}
 \tilde{h}_k(t,x,y)*\sum_{l=1}^{\infty}(-1)^{l}\tilde{g}_k^{*l}(t,x,y)=O(t^k)
\end{align*}
 by lemmata \ref{graph_estimate} and \ref{convolution_estimate}.
\end{proof}

\section{The asymptotics of the heat kernel}
%%%%%%%%%%%%%%%%%%%%%%%%%%%%%%%%%%%%%%%%%%%%%%%%%%%%%%%%%%%%%%%%%%%%%%%%%%%%%%%%%%%%%%%%%%%%%%%%%%%%%%%%%%%%%%%%%%%%%%%%
\label{section:asymptotics}

\begin{theorem}
\label{asymptotic_expansion}
 The heat kernel is approximated by the parametrix.
\begin{equation*}
 e(t,x,y)=\tilde{h}_k(t,x,y) + O(t^k)
\end{equation*}
Consequently, it admits an asymptotic expansion of the form
 \begin{align*}
  \int_G e(t,x,x) dx \sim_{t \ra 0^+} 
  \frac{1}{\sqrt{4\pi t}} \sum_{n=0}^{\infty}\int_G a_n(x) dx \cdot t^n 
 + \frac{1}{4} \sum_{v \in V}\sum_{\alpha \sim v} \sigma_v^{\alpha\alpha} \sum_{n=0, \frac{1}{2},1,\frac{3}{2}, \hdots}a_n^b(v) \cdot t^n
 \end{align*}
 with some universal coefficients $a_n(x)$ and $a_n^b(v)$.
\end{theorem}
\begin{proof}
This first claim follows from $\tilde{h}_k*\sum_{l=1}^{\infty}(-1)^{l}\tilde{g}_k^{*l}=O(t^k)$ by lemmata  \ref{graph_estimate} and \ref{convolution_estimate}. 

For the second one, we will split the proof into two parts, depending on whether $x$ is in the $V_{\frac{l_0}{3}}^v$ neighbourhood of a vertex or not. 

If $x$ is in a $V^v_{\frac{l_0}{3}}$ neighbourhood on the edge $\alpha$ we can unwind the definition of the parametrix all the way to the real line. Wer will denote the $u$-function in the coordinates of the star graph around $v$ by $u^v$.
\begin{align*}
& \int_{V^v_{\frac{l_0}{3}}} e(t,x,x)dx \\
=&  \sum_{\alpha \sim v}\int_0^{\frac{l_0}{3}}h_k^v(t,x_{\alpha},x_{\alpha})dx_{\alpha}+O(t^k)\\
=& \sum_{\alpha \sim v}\int_0^{\frac{l_0}{3}}h_k^{\alpha \alpha}(t,x_{\alpha},x_{\alpha})
+\sigma_v^{\alpha \alpha}h_k^{\alpha \alpha}(t,x_{\alpha},-x_{\alpha})dx_{\alpha}+O(t^{k})\\
=&\sum_{\alpha \sim v}\int_0^{\frac{l_0}{3}} \frac{1}{\sqrt{4\pi t}}\sum_{l=0}^{k}u_l^v(x_{\alpha},x_{\alpha})t^l dx_{\alpha}\\
&+\sum_{\alpha \sim v}\sigma_v^{\alpha \alpha}\int_0^{\frac{l_0}{3}}\frac{1}{\sqrt{4\pi t}}e^{-\frac{x_{\alpha}^2}{t}}\sum_{l=0}^{k}u_l^v(x_{\alpha},-x_{\alpha})t^ldx_{\alpha}+O(t^k)\\
=& \frac{1}{\sqrt{4\pi t}}\sum_{\alpha \sim v}\sum_{l=0}^{k}\int_0^{\frac{l_0}{3}}u_l^v(x_{\alpha},x_{\alpha}) dx_{\alpha}t^l\\
&+\frac{1}{\sqrt{4\pi }}\sum_{\alpha \sim v}\sigma_v^{\alpha \alpha}\sum_{l=0}^{k}\int_0^{\frac{l_0}{3}t^{-\frac{1}{2}}}e^{-x_{\alpha}^2}u_l^v(t^{\frac{1}{2}}x_{\alpha},-t^{\frac{1}{2}}x_{\alpha})dx_{\alpha}t^{l}+O(t^k)
\end{align*}
Now we will treat the part of the integral over $V^c_{\frac{l_0}{3}}$, assume that $x$ is on the edge $\alpha$ with end vertices $v$ and $v'$, then we have
\begin{align*}
 &e(t,x_{\alpha}, x_{\alpha}) \\
 =& \chi_v(x_{\alpha})h_k^v(t,x_{\alpha}, x_{\alpha}) + \chi_{v'}(x_{\alpha})h_k^{v'}(t,x_{\alpha}, x_{\alpha}) + O(t^k)\\
 =& h_k^{\alpha \alpha}(t,x_{\alpha}, x_{\alpha}) + O(t^k)
\end{align*}
by the properties of the partition of unity $\chi$ and the fact that $h_k(t,x,y)=O(t^{\infty})$ for $x$ and $y$ bounded away from each other by lemma \ref{real_parametrix}. Note that the part from $h_k^{v'}$ would be parametrised in the opposite direction but we can simply write everything in the same coordinates, this works because the $u$-functions are local, see \ref{local_u-functions}. 
This means
\begin{align*}
&\int_{V^c_{\frac{l_0}{3}}}e(t,x,x) dx \\
=& \sum_{\alpha \in E}\int_{\frac{l_0}{3}}^{l(\alpha)-\frac{l_0}{3}} h_k^{\alpha \alpha}(t,x_{\alpha},x_{\alpha})dx_{\alpha}+O(t^k)\\
=& \sum_{\alpha \in E}\int_{\frac{l_0}{3}}^{l(\alpha)-\frac{l_0}{3}}\frac{1}{\sqrt{4\pi t}} \sum_{l=0}^{k}u_l(x_{\alpha},x_{\alpha}) t^l dx_{\alpha} +O(t^k)
\end{align*}
Putting the two parts together gives
 \begin{align*}
 & \int_G e(t,x,x) dx \\
  \sim_{t \ra 0^+} &
  \frac{1}{\sqrt{4\pi t}} \sum_{l=0}^{k}\int_G u_l(x,x) dx \cdot t^l \\
& + \frac{1}{4} \sum_{v \in V}\sum_{\alpha \sim v} \sigma_v^{\alpha\alpha} \sum_{l=0}^{k}\int_0^{\frac{l_0}{3}t^{-\frac{1}{2}}}e^{-x_{\alpha}^2}u_l^v(t^{\frac{1}{2}}x_{\alpha},-t^{\frac{1}{2}}x_{\alpha}) dx_{\alpha}\cdot t^{l} +O(t^k)
 \end{align*}
As each $u_l(x,x)$ in a coordinate patch is local there is a well defined global function $u_l(x,x)$, so this implies that $a_l(x)=u_l(x,x)$. 
To compute the coefficients at the vertices one needs to plug in the Taylor expansions of $u_l^v(x,-x)$ for $x$ close to zero. We have
 \begin{align}
 \label{boundary_coefficients}
  \sum_{l=0, \frac{1}{2},1,\frac{3}{2}, \hdots}a_l^b(v) t^l = \frac{2}{\sqrt{\pi}}\sum_{\alpha \sim v} \sigma_v^{\alpha\alpha}\sum_{l=0}^{k}\int_0^{\frac{l_0}{3}t^{-\frac{1}{2}}}e^{-x_{\alpha}^2}u_l^v(t^{\frac{1}{2}}x_{\alpha},-t^{\frac{1}{2}}x_{\alpha}) dx_{\alpha}t^{l} +O(t^k)
 \end{align}
 We will now compute some $u$-functions and then use this equation to find the first few coefficients.
\end{proof}

\subsection{The $u$-functions}
%%%%%%%%%%%%%%%%%%%%%%%%%%%%%%%%%%%%%%%%%%%%%%%%%%%%%%%%%%%%%%%%%%%%%%%%%%%%%%%%%%%%%%%%%%%%%%%%%%%%%%%%%%%%%%%%%%%%%%%%

In order to compute the coefficients in the asymptotics, we will compute the first few $u$-functions explicitly. First, we will need the values on the diagonal $u_l(x,x)$, and second we need the Taylor expansion of $u_l(x,-x)$ for $x$ approaching zero, that is a vertex of the graph.

\begin{lemma}
\label{u_diagonal}
On the diagonal the first terms are
 \begin{align*}
u_0(x,x)=&1\\
u_1(x,x)=&U(x)\\
u_2(x,x)=&\frac{U''(x)}{6}+\frac{U(x)^2}{2}\\
u_3(x,x) =& \frac{1}{60}U^{(4)}(x)+\frac{1}{6}U''(x)U(x)+ \frac{1}{12}U'(x)^2+\frac{1}{6}U(x)^3
 \end{align*}
 \end{lemma}
\begin{proof}
The $u$-functions are local and smooth by our assumptions on the potential. Thus we can carry out this computation globally and coordinate free.
Assume that $x$ and $y$ are on the same edge.
From the transport equation, lemma \ref{transport_equation}, we have
\begin{equation*}
u_1(x,y) =(y-x)^{-1}\int_x^y U(z)dz
\end{equation*}
We will plug in the Taylor expansion of $U$ at the point $x$. This and all following equations are meant to hold for any finite Taylor expansion up to a suitable remainder term.
\begin{align*}
u_1(x,y) &=(y-x)^{-1}\sum_{l\ge 0}\int_x^y (z-x)^{l}dz\frac{U^{(l)}(x)}{l!}\\
&=\sum_{l \ge 0}(y-x)^{l}\frac{U^{(l)}(x)}{(l+1)!}
\end{align*}
This gives the value of $u_1(x,x)$ and 
\begin{align*}
 D_yu_1(x,y) = -\sum_{l\ge 2}(y-x)^{l-2}\frac{U^{(l)}(x)}{(l-2)!(l+1)} - U(y) \sum_{l \ge 0}(y-x)^l\frac{U^{(l)}(x)}{(l+1)!}
\end{align*}
which we now use to compute $u_2(x,y)$, valid for $y$ close to $x$.
 \begin{align*}
  u_2(x,y)=&-(y-x)^{-2}\int_x^y (z-x)D_zu_1(x,z)dz\\
  =& (y-x)^{-2}\int_x^y \sum_{l \ge 2}(z-x)^{l-1}\frac{U^{(l)}(x)}{(l-2)!(l+1)}dz\\
 & +(y-x)^{-2}\int_x^y U(z)\sum_{l\ge 0}(z-x)^{l+1}\frac{U^{(l)}(x)}{(l+1)!}dz\\
  =& (y-x)^{-2}\sum_{l \ge 2}\frac{(y-x)^{l}}{l}\frac{U^{(l)}(x)}{(l-2)!(l+1)}\\
 & +(y-x)^{-2}\int_x^y \sum_{m \ge 0}\frac{U^{(m)}(x)}{m!}(z-x)^m\sum_{l \ge 0}(z-x)^{l+1}\frac{U^{(l)}(x)}{(l+1)!}dz\\
  =& \sum_{l \ge 2}\frac{U^{(l)}(x)}{(l-2)!l(l+1)}(y-x)^{l-2}\\
 & +\sum_{l,m \ge 0} \frac{U^{(l)}(x)U^{(m)}(x)}{(l+1)!m!(m+l+2)}(y-x)^{m+l}
 \end{align*}
This determines $u_2(x,x)$ and we have
 \begin{align*}
  \partial_y^2 u_2(x,y)
  =& \sum_{l \ge 4}\frac{U^{(l)}(x)}{(l-4)!l(l+1)}(y-x)^{l-4}\\
 & +\sum_{l+m\ge 2} \frac{U^{(l)}(x)U^{(m)}(x)(m+l)(m+l-1)}{(l+1)!m!(m+l+2)}(y-x)^{m+l-2}
 \end{align*}
Finally, from the transport equation, lemma \ref{transport_equation}
 \begin{align*}
  u_3(x,x)=& -\frac{1}{3}D_yu_2(x,x) \\
  =&\frac{1}{3}\partial_y^2 u_2(x,x) + \frac{1}{3}U(x)u_2(x,x)\\
   =&\frac{1}{60}U^{(4)}(x)+\frac{1}{6}U''(x)U(x)+ \frac{1}{12}U'(x)^2+\frac{1}{6}U(x)^3
 \end{align*}
\end{proof}

\begin{lemma}
\label{u1u2_limits}
 The Taylor expansions of $u_1^v(x,-x)$ and $u_2^v(x,-x)$ for small $x$ are given by
 \begin{align*}
  u_1^v(x,-x) = & U(v) + \frac{1}{6}U''(v)x^2 + \frac{1}{5!}U^{(4)}(v)x^4 +O(x^6)\\
  u_2^v(x,-x) = & \frac{1}{6}U''(v)  + \frac{1}{2}U(v)^2+ \left(\frac{1}{60}U^{(4)}(v) + \frac{1}{6}U(v)U''(v)\right)x^2 +O(x^4)
 \end{align*}
where as before we assume that the odd derivatives of $U$ vanish at the vertices.
\end{lemma}
\begin{proof}
Recall that in the coordinates of a star graph $0$ corresponds to the central vertex and all edges are parametrised away from the vertex. Using the Taylor expansion of $U$ at zero gives
\begin{align*}
u_1(x,y)=&-(y-x)^{-1}\sum_{l \ge 0}(x^{l+1}-y^{l+1})\frac{U^{(l)}(0)}{(l+1)!}\\
=&\sum_{l \ge 0}\frac{U^{(l)}(0)}{(l+1)!}\sum_{m=0}^lx^my^{l-m}
\end{align*}
As in the proof of lemma \ref{u_diagonal} this holds for any finite Taylor expansion up to a suitable remainder term.
This determines the Taylor expansion of $u_1(x,-x)$. 
\begin{align*}
D_yu_1(x,y) =& -\sum_{l \ge 2}\frac{U^{(l)}(0)}{(l+1)!}\sum_{m=0}^{l-2}(l-m)(l-m-1)x^ly^{l-m-2}\\
& + \sum_{l,m \ge 0}\frac{U^{(l)}(0)U^{(m)}(0)}{l!(m+1)!}(y-x)^{-1}y^l(x^{m+1}-y^{m+1})
\end{align*}

We will now use this expression to compute the expansion of $u_2(x,-x)$ for $x$ close to zero.
\begin{align*}
 u_2(x,-x) =& \frac{1}{4x^2}\int_{-x}^x(y-x)D_y u_1(x,y) dy\\
 =& -\frac{1}{4x^2}\int_{-x}^x(y-x)\sum_{l \ge 2}\frac{U^{(l)}(0)}{(l+1)!}\sum_{m=0}^{l-2}(l-m)(l-m-1)x^ly^{l-m-2} dy\\
 &+\frac{1}{4x^2}\int_{-x}^x \sum_{l,m \ge 0}\frac{U^{(l)}(0)U^{(m)}(0)}{l!(m+1)!}y^l(x^{m+1}-y^{m+1})  dy\\
 = & \frac{1}{6}U''(0) + \frac{1}{60}U^{(4)}(0)x^2 + \frac{1}{2}U(0)^2 + \frac{1}{6}U(0)U''(0)x^2 +O(x^4)
\end{align*}

\end{proof}

\subsection{Computing the coefficients}
%%%%%%%%%%%%%%%%%%%%%%%%%%%%%%%%%%%%%%%%%%%%%%%%%%%%%%%%%%%%%%%%%%%%%%%%%%%%%%%%%%%%%%%%%%%%%%%%%%%%%%%%%%%%%%%%%%%%%

Using the approximation of the heat kernel by the parametrix and the values of the $u$-functions we can now compute the values of the coefficients in the asymptotics.

\begin{theorem}
\label{coefficients_heat_trace}
With our assumptions on the potential $a_n^b(v)=0$ whenever $n$ is not an integer. The first few non-zero coefficients are as follows.
\begin{align*}
& a_0(x) = 1 && a_0^b(v)=1\\
& a_1(x) = U(x) && a_1^b(v)=U(v)\\
& a_2(x) =  \frac{1}{6}U''(x)+\frac{1}{2}U(x)^2 && a_2^b(v)=\frac{1}{4}U''(v)+\frac{1}{2}U(v)^2\\
& a_3(x) = \frac{1}{60}U^{(4)}(x)+\frac{1}{6}U''(x)U(x)+ \frac{1}{12}U'(x)^2+\frac{1}{6}U(x)^3  \\
& a_3^b(v)=\frac{1}{32}U^{(4)}(v)+\frac{1}{4}U(v)U''(v)+\frac{1}{6}U(v)^3
\end{align*}
\end{theorem}

\begin{proof}
We showed $a_l(x)=u_l(x,x)$ in the proof of theorem \ref{asymptotic_expansion}, so the values for $a_0(x), \hdots, a_3(x)$ follow directly from lemma \ref{u_diagonal}.

To compute the boundary coefficients we will use the Taylor expansions computed in lemma \ref{u1u2_limits} and plug them in equation (\ref{boundary_coefficients}) from the proof of theorem \ref{asymptotic_expansion}. Note that because of our assumptions on the potential, only the even terms in the Taylor expansions appear, so all terms at half powers of $t$ vanish.
\begin{align*}
 &\sum_{l=0, \frac{1}{2},1,\frac{3}{2}, \hdots}a_l^b(v) t^l \\
 =& \frac{2}{\sqrt{\pi}}\sum_{\alpha \sim v} \sigma_v^{\alpha\alpha}\sum_{l=0}^{k}\int_0^{\frac{l_0}{3}t^{-\frac{1}{2}}}e^{-x_{\alpha}^2}u^v_l(t^{\frac{1}{2}}x_{\alpha},-t^{\frac{1}{2}}x_{\alpha})t^{l} dx_{\alpha} +O(t^k)\\
=& \frac{2}{\sqrt{\pi }}\sum_{\alpha \sim v}\sigma_v^{\alpha \alpha} \int_0^{\frac{l_0}{3}t^{-\frac{1}{2}}}e^{-x_{\alpha}^2}dx_{\alpha}\\
 & +\frac{2}{\sqrt{\pi }}\sum_{\alpha \sim v}\sigma_v^{\alpha \alpha}t \int_0^{\frac{l_0}{3}t^{-\frac{1}{2}}}e^{-x_{\alpha}^2}\left(U(v)+\frac{1}{6}U''(v)tx_{\alpha}^2+\frac{1}{5!}U^{(4)}(v)t^2x_{\alpha}^4\right)dx_{\alpha}\\
 &+ \frac{2}{\sqrt{\pi }}\sum_{\alpha \sim v}\sigma_v^{\alpha \alpha} t^2\int_0^{\frac{l_0}{3}t^{-\frac{1}{2}}}e^{-x_{\alpha}^2}\left(\frac{1}{2}U(v)^2+\frac{1}{6}U''(v)+\frac{1}{60}U^{(4)}(v)tx_{\alpha}^2 + \frac{1}{6}U(v)U''(v)tx_{\alpha}^2\right)dx_{\alpha}\\
&+ \frac{2}{\sqrt{\pi }}\sum_{\alpha \sim v}\sigma_v^{\alpha \alpha} t^3\int_0^{\frac{l_0}{3}t^{-\frac{1}{2}}}e^{-x_{\alpha}^2}\left(\frac{1}{60}U^{(4)}(v)+\frac{1}{6}U''(v)U(v)+\frac{1}{6}U(v)^3\right)dx_{\alpha} + O(t^4)\\
=&\sum_{\alpha \sim v}\sigma_v^{\alpha \alpha}\left(\begin{array}{c c}1+tU(v)+t^2\left(\frac{U''(v)}{4}+\frac{U(v)^2}{2}\right)\\
 + t^3\left(\frac{1}{32}U^{(4)}(v)+\frac{1}{4}U(v)U''(v)+\frac{1}{6}U(v)^3 \right) +  O(t^4) \end{array}\right) 
\end{align*}
Here we used 
\begin{align*}
 \int_0^{\frac{l_0}{3}t^{-\frac{1}{2}}}e^{-x^2}dx =& \frac{\sqrt{\pi}}{2}+O(t^{\infty})\\
 \int_0^{\frac{l_0}{3}t^{-\frac{1}{2}}}e^{-x^2}x^2dx =& \frac{\sqrt{\pi}}{4}+O(t^{\infty})\\
 \int_0^{\frac{l_0}{3}t^{-\frac{1}{2}}}e^{-x^2}x^4dx =& \frac{3}{8}\sqrt{\pi}+O(t^{\infty})
\end{align*}

\end{proof}

\begin{remark}
One can compare the coefficients computed in theorem \ref{coefficients_heat_trace} with results for manifolds in \cite{Gilkey79} on the unit interval and the coefficients match. The manifold expansion contains boundary coefficients at half powers of $t$ but these vanish in our setting because of the assumptions we made on the potential.  
\end{remark}

\section*{Acknowledgement}
%%%%%%%%%%%%%%%%%%%%%%%%%%%%%%%%%%%%%%%%%%%%%%%%%%%%%%%%%%%%%%%%%%%%%%%%%%%%%%%%%%%%%%%%%%%%%%%%%%%%%%%%%%%%%%%%%%%%%%%%%%%%%

This work was made possible through a grant from the DFG, that allowed me to spend a year at Royal Holloway to work on this project. It is a pleasure to thank Jens Bolte for coming up with the topic and multiple fruitful discussions throughout my work on this.

%\printindex
\bibliographystyle{amsalpha}
% \appendix
\addcontentsline{toc}{section}{References}
\bibliography{/home/ralf/personalfiles/phd/literatur}
% \begin{thebibliography}{A}
% \end{thebibliography}

\end{document}